\documentclass[12pt]{article}

\usepackage{amsmath}
\usepackage{amssymb}
\usepackage[mathscr]{eucal}
\usepackage{graphicx}
\usepackage{color}

\usepackage{bm}
\usepackage{mathrsfs}
\usepackage{amsthm}
\usepackage{enumerate}
\usepackage[mathscr]{eucal}
\usepackage{graphicx}

\newtheorem{Theorem}{\sc Theorem}%%%[section]
\newtheorem{Definition}[Theorem]{\sc Definition}

\newcommand{\R}{{\if mm {\rm I}\mkern -3mu{\rm R}\else \leavevmode
		\hbox{I}\kern -.17em\hbox{R} \fi}}

\def\sqr#1#2{{
		\vcenter{
			\vbox{\hrule height.#2pt
				\hbox{\vrule width.#2pt height#1pt \kern#1pt
					\vrule width.#2pt
				}
				\hrule height.#2pt
			}
		}
}}

\def\real{\mathbb{R}}

\def\lista#1
{{ \itemindent 0.0cm \labelsep .2cm \leftmargin 0.8cm \rightmargin
		0.0cm \labelwidth 0.6cm \topsep 0.0mm
		\parsep 0.0mm
		\itemsep 0.0mm
		\begin{list}{}
			{ \setlength{\leftmargin}{.8cm} \setlength{\rightmargin}{0.0cm}
				\setlength{\parsep}{0.0mm} \setlength{\topsep}{.0mm}
				\setlength{\parskip}{.0cm} \setlength{\itemsep}{.0cm} }
			{#1}\end{list}} }

\begin{document}
\title{
Distributed optimal control problems for a class of elliptic hemivariational inequalities with a parameter and its asymptotic behavior}

\author{
Claudia M. Gariboldi  \footnote{\, Depto. Matem\'atica, FCEFQyN, Universidad Nacional de R\'io Cuarto, Ruta 36 Km 601,
5800 R\'io Cuarto, Argentina. E-mail: cgariboldi@exa.unrc.edu.ar.}
\hspace{.15cm} and \ Domingo A. Tarzia \footnote{\, Depto. Matem\'atica, FCE, Universidad Austral, Paraguay 1950, S2000FZF Rosario, Argentina.} \footnote{\, CONICET, Argentina. E-mail: DTarzia@austral.edu.ar.}
}

\date{}

\maketitle
	
\medskip

\vspace{-.4cm}

Dedicated to Professor Stanislaw Mig\'orski on the occasion of his 60th birthday

\medskip

\vspace{.4cm}

\noindent {\bf Abstract.} \

In this paper, we study optimal control problems on the internal energy for a system governed by a class of elliptic boundary hemivariational inequalities with a parameter. The system has been originated by a steady-state heat conduction problem with non-monotone multivalued subdifferential boundary condition on a portion of the boundary of the domain described by the Clarke generalized gradient of a locally Lipschitz function. We prove an existence result for the optimal controls and we show an asymptotic result for the optimal controls and the system states, when the parameter, like a heat transfer coefficient, tends to infinity on a portion of the boundary.

\medskip
	
	
%\smallskip

\medskip
	

\medskip

{\thispagestyle{empty}} %NO BORRE ESTA LINEA

%\newpage

%%===============================
\section{Introduction}
%%===============================

We consider a bounded domain $\Omega$ in $\real^d$ whose
regular boundary $\Gamma $ consists of the union of three disjoint portions $\Gamma_{i}$, $i=1$, $2$, $3$
%, $\Gamma_{2}$ and $\Gamma_{3}$
with $|\Gamma_{i}|>0$, where $|\Gamma_i|$ denotes the $(d-1)$-dimensional
Hausdorff measure of the portion $\Gamma_i$ on $\Gamma$.
The outward normal vector on
the boundary is denoted by $n$.
We formulate the following steady-state heat conduction problem
%$S$ and $S_{\alpha }$
with mixed boundary conditions ~\cite{AK, BBP, G, LCB, Ta2, Ta3}:
\begin{eqnarray}
&&
-\Delta u=g \ \ \mbox{in} \ \ \Omega,
\ \ \quad u\big|_{\Gamma_{1}}=0,
\ \ \quad-\frac{\partial u}{\partial n}\big|_{\Gamma_{2}}=q,
\ \ \quad u\big|_{\Gamma_{3}}=b,
\label{P}
\end{eqnarray}
where $u$ is the temperature in $\Omega$,
$g$ is the internal energy in $\Omega$,
$b$ is the temperature on $\Gamma_{3}$ and
$q$ is the heat flux on $\Gamma_{2}$,
which satisfy the hypothesis:
$g\in H=L^2(\Omega)$,
$q\in Q=L^2(\Gamma_2)$ and
$b\in H^{\frac{1}{2}}(\Gamma_3)$.

Throughout the paper we use the following notation
\begin{eqnarray*}
&& V=H^{1}(\Omega),
\quad
V_{0}=\{v\in V \mid v = 0 \ \ \mbox{on} \ \ \Gamma_{1} \},
\\[2mm]
&&
K=\{v\in V \mid
v = 0 \ \ \mbox{on} \ \ \Gamma_{1},\
v = b \ \ \mbox{on} \ \ \Gamma_{3} \},
\quad
K_{0}=\{v\in V \mid
v = 0 \ \ \mbox{on} \ \ \Gamma_{1}\cup \Gamma_3 \}, \\[2mm]
%&&
%H=L^2(\Omega), \quad
%Q=L^2(\Gamma_2), \quad
%(g,h)_{H}=\int_{\Omega}g h\,dx,
%\quad (q,\eta)_{Q}=\int_{\Gamma_{2}}q\eta \,d\Gamma, \\
&&
a(u,v)=\int_{\Omega }\nabla u \, \nabla v \, dx, \quad L(v)= \int_{\Omega}g v \,dx -
\int_{\Gamma_{2}}q \gamma (v) \,d\Gamma,
%%(g,v)_{H}-(q,\gamma_{0}(v))_{Q},
\end{eqnarray*}
where $\gamma \colon V \to L^2(\Gamma)$
denotes the trace operator on $\Gamma$.
In what follows, we write $u$ for the trace of a function $u \in V$ on the boundary.
In a standard way, we obtain the following variational formulation of (\ref{P}):
\begin{eqnarray}
&&
\hspace{-1cm}
\mbox{find} \ \ u_{\infty}\in K \ \ \mbox{such that}\ \
a(u_{\infty},v)=L(v)
\ \ \mbox{for all} \ \ v\in K_{0},
\label{Pvariacional}
\end{eqnarray}
The standard norms on $V$ and $V_0$ are denoted by
\begin{eqnarray*}
&&
\| v \|_V = \Big(
\| v \|^2_{L^2(\Omega)}
+ \| \nabla v \|^2_{L^2(\Omega;\real^d)} \Big)^{1/2}
\ \ \mbox{for} \ \ v \in V, \\ [2mm]
&&
\| v \|_{V_0} = \| \nabla v \|_{L^2(\Omega;\real^d)}
\ \ \mbox{for} \ \ v \in V_0.
\end{eqnarray*}
It is well known by the Poincar\'e inequality,
see~\cite{CLM, R, Ta2}, that on $V_0$ the above two norms
are equivalent. Note that the form $a$ is bilinear, symmetric, continuous and coercive with constant $m_a > 0$, i.e.
\begin{equation}\label{coercive}
%\exists \lambda_{1}>0\quad \text{such that}\quad
a(v, v) = \|v\|^{2}_{V_0} \ge m_a \|v\|^{2}_{V}
\ \ \mbox{for all} \ \ v\in V_{0}.
\end{equation}

We remark that, under additional hypotheses on the data
$g$, $q$ and $b$, problem (\ref{P}) can be considered as steady-state two-phase Stefan problem, see, for example,~\cite{GT,TT,Ta, Ta3}. We can particularly see it in ~\cite{GT} (Example 1 in page 629, Example 2 in page 630, and Example 3 in page 631); in ~\cite{TT} (Example (i) and (ii) in psge 35, and Example (iii) in page 36), and in ~\cite{Ta3} (Example 1 and Example 2 in page 180).

Now, in this paper, we consider the mixed nonlinear boundary value problem for an elliptic equation as follows:
\begin{equation}\label{Pjalfa}
-\Delta u=g \ \ \mbox{in} \ \ \Omega,
\ \quad u\big|_{\Gamma _{1}}=0,
\ \quad  -\frac{\partial u}{\partial n}\big|_{\Gamma_{2}}=q,  \ \quad -\frac{\partial u}{\partial n}\big|_{\Gamma_{3}}
\in \alpha \, \partial j(u),
\end{equation}
which has been recently studied in~\cite{GMOT}.

Here $\alpha$ is a positive constant which can be considered as the heat transfer coefficient on the boundary while
the function $j \colon \Gamma_{3} \times \real \to \real$, called a superpotential (nonconvex potential),
is such that $j(x, \cdot)$ locally Lipschitz for a.e. $x \in \Gamma_3$
and not necessary differentiable.
Since in general $j(x, \cdot)$ is nonconvex, so the multivalued condition on $\Gamma_3$ in problem (\ref{Pjalfa})
is described by a nonmonotone relation expressed by the generalized gradient of Clarke~\cite{C}.
Such multivalued relation in problem (\ref{Pjalfa}) is met
in certain types of steady-state heat conduction problems
(the behavior of a semipermeable membrane of finite
thickness, a temperature control problems, etc.).
Further, problem (\ref{Pjalfa}) can be considered as a prototype of several boundary semipermeability models,
see~\cite{MO,NP,P,ZLM}, which are motivated by problems arising in hydraulics, fluid flow problems through porous media,
and electrostatics, where the solution represents the pressure and the electric potentials.
Note that the analogous problems with maximal monotone multivalued boundary relations (that is the case when $j(x, \cdot)$
is a convex function) were considered in~\cite{Barbu,DL},
see also references therein.

Under the above notation,
the weak formulation of the elliptic problem (\ref{Pjalfa})
becomes the following elliptic boundary hemivariational inequality~\cite{GMOT}:
\begin{equation}\label{Pj0alfavariacional}
\mbox{find} \ \ u \in V_0 \ \ \mbox{such that} \ \
a(u,v) + \alpha \int_{\Gamma_{3}}j^{0}(u;v)\, d\Gamma
\geq L(v) \ \ \mbox{\rm for all} \ \  v\in V_{0}.
\end{equation}
Here and in what follows we often omit the variable
$x$ and we simply write $j(r)$ instead of $j(x, r)$.
The stationary heat conduction models with nonmonotone multivalued subdifferential interior and boundary semipermeability relations
cannot be described by convex potentials.
They use locally Lipschitz potentials and their weak formulations lead to hemivariational inequalities, see~\cite[Chapter~5.5.3]{NP} and~\cite{P}.

We mention that theory of hemivariational and variational inequalities has been proposed in the 1980s
by Panagiotopoulos, see~\cite{NP,P0,P1},
as variational formulations of important classes of inequality problems in mechanics.
In the last few years, new kinds of variational, hemivariational, and variational-hemivariational
inequalities have been investigated, see recent monographs~\cite{CLM,MOS,SM},
and the theory has emerged today as a new and interesting branch of applied mathematics.

We consider the distributed optimal control problem of the type studied in \cite{GaTa, Li, Tr} given by:
\begin{equation}\label{OPVariational}
\text{find}\quad g^{*}\in H \quad \text{such that} \quad J(g^{*})=\min_{g\in H}J(g)
\end{equation}
with
\begin{equation}
J(g)=\frac{1}{2}||u_{g}-z_{d}||^{2}_{H}+\frac{M}{2}||g||^{2}_{H}
\end{equation}
where $u_{g}$ is the unique solution to the variational equality (\ref{Pvariacional}), $z_{d}\in H$ given and $M$ a positive constant.

The goal of this paper is to formulate, for each $\alpha>0$, the following new distributed optimal control problem
\begin{equation}\label{OPHemivariational}
\text{find}\quad g_{\alpha}^{*}\in H \quad \text{such that} \quad J_{\alpha}(g_{\alpha}^{*})=\min_{g\in H}J_{\alpha}(g)
\end{equation}
with
\begin{equation}
J_{\alpha}(g)=\frac{1}{2}||u_{\alpha g}-z_{d}||^{2}_{H}+\frac{M}{2}||g||^{2}_{H}
\end{equation}
where $u_{\alpha g}$ is a solution to the hemivariational inequality (\ref{Pj0alfavariacional}), $z_{d}\in H$ given and $M$ a positive constant, and to study the convergent to problem (\ref{OPHemivariational}) when the parameter $\alpha$ goes to infinity.

The paper is structured as follows. In Section~\ref{Preliminaries} we establish preliminaries concepts of the hemivariational inequalities theory, which are necessary for the development of the following sections. In Section~\ref{OCP}, for each $\alpha >0$, we obtain an existence result of solution to the optimal control problem (\ref{OPHemivariational}). Finally, in Section~\ref{Asymptotic}, we prove the strong convergence of a sequence of optimal controls of the problems (\ref{OPHemivariational}) to the unique optimal control of the problem (\ref{OPVariational}), when the parameter $\alpha$ goes to infinity. Moreover, we obtain the strong convergence of the system states related to the problems (\ref{OPHemivariational}) to the system state related to the problem (\ref{OPVariational}), when $\alpha$ goes to infinity. These results generalize for a locally Lipschitz function $j$, under the hypothesis $H(j)$ and $(H_1)$, the classical results obtained in ~\cite{GaTa} for a quadratic superpotential $j$.

\section{Preliminaries}\label{Preliminaries}

In this section we recall standard notation and preliminary concepts, which are necessary for the development of this paper.

Let $(X, \| \cdot \|_{X})$ be a reflexive Banach space,  $X^{*}$ be its dual, and $\langle \cdot, \cdot \rangle$ denote the duality between $X^*$ and $X$.
For a real valued function defined
on $X$, we have the following
definitions~\cite[Section~2.1]{C} and~\cite{DMP,MOS}.
%%\cite[Definition~37, p.121]{SM}.
\begin{Definition}
	A function $\varphi \colon X\rightarrow \mathbb{R}$
	is said to be locally Lipschitz, if for every $x\in X$
	there exist $U_{x}$ a neighborhood of $x$ and a constant $L_{x}>0$ such that
	$$
	|\varphi(y)-\varphi(z)|\leq L_{x}\|y-z\|_{X}
	\ \ \mbox{\rm for all} \ \ y, z\in U_{x}.
	$$
	For such a function the generalized (Clarke) directional derivative of $j$ at the point $x\in X$ in the direction
	$v\in X$ is defined by
	$$
	\varphi^{0}(x;v)=\limsup\limits_{y \rightarrow x, \, \lambda \rightarrow 0^{+}}
	\frac{\varphi(y +\lambda v)-\varphi(y)}{\lambda} \, .
	$$
	The generalized gradient (subdifferential)
	of $\varphi$ at $x$ is a subset of the dual space $X^{*}$ given by
	$$
	\partial \varphi(x)=\{\zeta\in X^{*} \mid \varphi^{0}(x;v)\geq \langle
	\zeta,v\rangle \ \ \mbox{\rm for all} \ \  v \in X\}.
	$$
\end{Definition}

We consider the following hypothesis.

\medskip

\noindent
${\underline{H(j)}}$: $j\colon \Gamma_3 \times \real \to \real$ is such that

\smallskip

\noindent
\quad (a) $j(\cdot, r)$ is measurable for all $r \in \real$,

\smallskip

\noindent
\quad (b)
$j(x, \cdot)$ is locally Lipschitz for a.e. $x \in \Gamma_3$,

\smallskip

\noindent
\quad (c)
there exist $c_0$, $c_1 \ge 0$ such that
%\begin{equation*}
$| \partial j(x, r)| \le c_0 + c_1 |r|$
for all $r \in \real$, a.e. $x\in \Gamma_3$,
%\end{equation*}

\smallskip

\noindent
\quad (d)
$j^0(x, r; b-r) \le 0$ for all $r \in \real$,
a.e. $x \in \Gamma_{3}$ with a constant $b \in \real$.

\medskip

Note that the existence results for elliptic hemivariational inequalities can be found in several contributions, see~\cite{CLM,unified,ELAS,MOS,NP}. In \cite[Theorem~4]{GMOT}, the hypothesis $H(j)$(d) is considered in order to obtain existence of a solution to problem \eqref{Pj0alfavariacional}. Moreover, under this condition the authors have studied the asymptotic behavior  when $\alpha \to \infty$ (see \cite[Theorem~7]{GMOT}).

We note that, if the hypothesis $H(j)$(d) is replaced
by the relaxed monotonicity condition (see \cite[Remark~10]{GMOT} for details)
\begin{equation*}
(e)\qquad j^0(x, r; s-r) + j^0(x,s; r-s) \le m_j \, |r-s|^2
\end{equation*}
for all $r$, $s \in \real$, a.e. $x\in\Gamma_3$
with $m_j \ge 0$,
and the following smallness condition
$$
(f)\qquad m_a > \alpha \, m_j \| \gamma\|^2
$$
is assumed, then
%%where $\|\gamma\|$ denotes the norm of the trace operator,
problem (\ref{Pj0alfavariacional}) is uniquely solvable,
see~\cite[Lemma~20]{ELAS} for the proof.
However, this smallness condition is not suitable in the study to problem \eqref{Pj0alfavariacional} since for a sufficiently large value of $\alpha$, it is not satisfied. Finally, in \cite{GMOT} we can find several examples of locally Lipschitz (nondifferentiable and nonconvex) functions which satisfies the above hypotheses.

\section{Optimal control problems}\label{OCP}

We know, by \cite{GaTa}, that there exists a unique optimal pair $(g^{*},u_{g^{*}})\in H\times V_{0}$ of the distributed optimal control problem (\ref{OPVariational}).
Now, we pass to a result on existence of solution to the optimal control problem (\ref{OPHemivariational}) in which the system is governed by the hemivariational inequality (\ref{Pj0alfavariacional}).

\begin{Theorem}\label{existence}
For each $\alpha > 0$, if $H(j) (a)-(d)$ holds,
then the distributed optimal control problems (\ref{OPHemivariational})
has a solution.
\end{Theorem}
\begin{proof}
By definition, for each $\alpha >0$, the functional $J_{\alpha}$ is bounded from below. Next, taking into account that the hemivariational inequality (\ref{Pj0alfavariacional}) has solution (see \cite[Theorem 4]{GMOT}), for each $\alpha >0$ and each $g\in H$, we denote by  $T_{\alpha}(g)$ the set of solutions of (\ref{Pj0alfavariacional}) and we have that
\begin{equation}\label{min1}
m=\inf\{J_{\alpha}(g), g\in H, u_{\alpha g}\in T_{\alpha}(g)\}\geq 0.
\end{equation}
Let $g_{n}\in H$ be a minimizing sequence to (\ref{min1}) such that
\begin{equation}\label{min2}
m\leq J_{\alpha}(g_{n}) \leq m + \frac{1}{n}.
\end{equation}
Taking into account that the functional $J_{\alpha}$ satisfies $$\lim\limits_{||g||_{H}\rightarrow +\infty}J_{\alpha}(g)=+\infty$$
we obtain that there exists $C_{1} >0$ such that
\begin{equation}\label{cota1}
||g_{n}||_{H}\leq C_{1}.
\end{equation}
Moreover, we can prove that there exists $C_{2} >0$ such that
\begin{equation}\label{cota2}
||u_{\alpha g_{n}}||_{V_{0}}\leq C_{2}.
\end{equation}
In fact, let $u_\infty \in K$ be the solution to problem (\ref{Pvariacional}). We have
\begin{equation*}\begin{split}
a(u_{\alpha g_{n}}, u_\infty -u_{\alpha g_{n}})
+ \alpha \int_{\Gamma_{3}}
j^0(u_{\alpha g_{n}}; u_\infty-u_{\alpha g_{n}}) \, d\Gamma
&\ge \int_{\Omega}g_{n}(u_\infty-u_{\alpha g_{n}})\, dx \\
&-\int_{\Gamma_{2}}q (u_\infty-u_{\alpha g_{n}})\, d\Gamma.\end{split}
\end{equation*}
Hence
\begin{equation*}\begin{split}
a(u_\infty - u_{\alpha g_{n}}, u_\infty -u_{\alpha g_{n}})
& \le a(u_\infty, u_\infty-u_{\alpha g_{n}})
+ \alpha \int_{\Gamma_{3}}j^0(u_{\alpha g_{n}}; b-u_{\alpha g_{n}}) \, d\Gamma \\
& + \int_{\Omega}g_{n} (u_{\alpha g_{n}}-u_\infty)\, dx-\int_{\Gamma_{2}}q (u_\infty-u_{\alpha g_{n}})\, d\Gamma.
\end{split}\end{equation*}
From hypothesis $H(j)$(d),
since the form $a$ is bounded (with positive constant $M_{a}$), we get
\begin{equation*}\begin{split}
\| u_\infty - u_{\alpha g_{n}} \|_{V_{0}}^2
& \le a(u_\infty, u_\infty-u_{\alpha g_{n}}) + \int_{\Omega}g_{n} (u_{\alpha g_{n}}-u_\infty)\, dx-\int_{\Gamma_{2}}q (u_\infty-u_{\alpha g_{n}})\, d\Gamma \\
& \le
M_{a} \| u_\infty \|_V \| u_\infty - u_{\alpha g_{n}} \|_V
+ \left(||g_{n}||_{H}+||q||_{Q}||\gamma||\right) \| u_\infty - u_{\alpha g_{n}} \|_V\\
&\le  \left( M_{a}\| u_\infty \|_V+C_{1}+C_{3}||q||_{Q}||\gamma||\right) \| u_\infty - u_{\alpha g_{n}} \|_{V_{0}}
\end{split}\end{equation*}
where $||\gamma||$ denote the norm of trace operator and $C_{3}$ is a positive constant due to the equivalence of norms. Subsequently, we obtain (\ref{cota2}). Therefore, there exist $f\in H$ and $\eta_{\alpha}\in V_{0}$ such that
\[
u_{\alpha g_{n}}\rightharpoonup \eta_{\alpha}\quad \text{in}\quad V_{0} \quad\text{weakly} \qquad \text{and}\qquad
 g_{n}\rightharpoonup f\quad \text{in}\quad H \quad \text{weakly}.
\]
Now, for all $g_{n}\in H$, we have
\[
a(u_{\alpha g_{n}},v) + \alpha \int_{\Gamma_{3}}j^{0}(u_{\alpha g_{n}};v)\, d\Gamma
\geq \int_{\Omega}g_{n}v \, dx -\int_{\Gamma_{2}}qv\, d\Gamma \ \ \mbox{\rm for all} \ \  v\in V_{0}
\]
and taking the upper limit, we obtain
\begin{equation}\label{CCC}
a(\eta_{\alpha}, v)
+ \alpha \limsup_{n\rightarrow +\infty} \int_{\Gamma_{3}}
j^0(u_{\alpha g_{n}}; v) \, d\Gamma \ge \int_{\Omega}fv \, dx -\int_{\Gamma_{2}}qv\, d\Gamma
\ \ \mbox{\rm for all} \ \ v \in V_0.
\end{equation}
By the compactness of the trace operator from $V$ into $L^2(\Gamma_{3})$, we have
$u_{\alpha g_{n}} \big|_{\Gamma_{3}} \to \eta_{\alpha} \big|_{\Gamma_{3}}$ in $L^2(\Gamma_3)$, as $n \to +\infty$,
and at least for a subsequence,
$u_{\alpha g_{n}}(x) \to \eta_{\alpha}(x)$ for a.e. $x \in \Gamma_3$
and $|u_{\alpha g_{n}}(x)| \le h_{\alpha}(x)$ a.e. $x \in \Gamma_3$,
where $h_{\alpha} \in L^2(\Gamma_3)$.
Since the function
$\real\times \real \ni (r, s) \mapsto j^0(x, r; s) \in \real$
a.e. on $\Gamma_3$ is upper semicontinuous,  see \cite[Proposition 3]{GMOT}, we obtain
$$
\limsup_{n\rightarrow +\infty} j^0(x,u_{\alpha g_{n}}(x); v(x)) \le j^0(x, \eta_{\alpha}(x); v(x))
\ \ \mbox{a.e.} \ \ x \in \Gamma_3.
$$
Next, from $H(j)(c)$, we deduce the estimate
\begin{equation*}
|j^0(x, u_{\alpha g_{n}}(x); v(x))|
\le (c_0 + c_1 |u_{\alpha g_{n}}(x)|) \, |v(x)| \le k_{\alpha}(x)
\ \ \mbox{a.e.} \ \ x \in \Gamma_3
\end{equation*}
where $k_{\alpha} \in L^1(\Gamma_3)$,
$k_{\alpha}(x) = (c_0 + c_1 h_{\alpha}(x)) |v(x)|$ and we apply the dominated convergence theorem, see~\cite{DMP} to get
$$
\limsup_{n\rightarrow +\infty} \int_{\Gamma_3} j^0(u_{\alpha g_{n}}; v) \, d\Gamma
\le \int_{\Gamma_3} \limsup_{n\rightarrow +\infty} j^0(u_{\alpha g_{n}}; v) \, d\Gamma
\le \int_{\Gamma_3} j^0(\eta_{\alpha}; v)\, d\Gamma.
$$
Using the latter in (\ref{CCC}), we obtain
\[
a(\eta_{\alpha},v) + \alpha \int_{\Gamma_{3}}j^{0}(\eta_{\alpha};v)\, d\Gamma
\geq \int_{\Omega}fv \, dx -\int_{\Gamma_{2}}qv\, d\Gamma \ \ \mbox{\rm for all} \ \  v\in V_{0}
\]
that is, $\eta_{\alpha}\in V_{0}$ is a solution to the hemivariational inequality (\ref{Pj0alfavariacional}). Next,  we have proved that
\[
\eta_{\alpha}=u_{\alpha f}
\]
where $u_{\alpha f}$ is a solution of the hemivariational inequality (\ref{Pj0alfavariacional}) for data $f\in H$ and $q\in Q$. Finally, from (\ref{min2}) and the weak lower semicontinuity  of $J_{\alpha}$, we have
\begin{equation}\begin{split}
m&\geq \liminf_{n\rightarrow +\infty}J_{\alpha}(g_{n})\\& \geq \frac{1}{2}\liminf_{n\rightarrow +\infty} ||u_{\alpha g_{n}}-z_{d}||^{2}_{H}+\frac{M}{2}\liminf_{n\rightarrow +\infty}||g_{n}||^{2}_{H}\\
& \geq \frac{1}{2}||u_{\alpha f}-z_{d}||^{2}_{H}+\frac{M}{2}||f||^{2}_{H}=J_{\alpha}(f),\nonumber
\end{split}\end{equation}
and therefore, $(f,u_{\alpha f})$ is an optimal pair to optimal control problem (\ref{OPHemivariational}).
\end{proof}

\section{Asymptotic behavior of the optimal controls}\label{Asymptotic}
%%as $\alpha \to \infty$}

In this section we investigate the asymptotic behavior of the optimal  solutions to problem~(\ref{OPHemivariational}) when $\alpha \rightarrow\infty$.
To this end, we need the following additional hypothesis
on the superpotential~$j$.

\medskip

\noindent
${\underline{(H_1)}}$: \quad
if $j^0(x, r; b-r) = 0$ for all $r \in \real$,
a.e. $x \in \Gamma_{3}$, then $r = b$.

\begin{Theorem}\label{Theorem6}
Assume $H(j)$ and $(H_1)$. If $(g_{\alpha},u_{\alpha g_{\alpha}})$ is a optimal solution to problem (\ref{OPHemivariational}) and $(g^{*},u_{\infty g^{*}})$ is the unique solution to problem (\ref{OPVariational}), then $g_{\alpha} \to g^{*}$ in $H$ strongly and $u_{\alpha g_{\alpha}} \to u_{\infty g^{*}}$ in $V$ strongly, when $\alpha \to \infty$.
\end{Theorem}

\begin{proof} We will make the prove in three steps.
\newline
\textbf{Step 1.} Since $(g_{\alpha},u_{\alpha g_{\alpha}})$ is a optimal solution to problem (\ref{OPHemivariational}), we have the following inequality
\[
\frac{1}{2}||u_{\alpha g_{\alpha}}-z_{d}||^{2}_{H}+\frac{M}{2}||g_{\alpha}||^{2}_{H}\leq \frac{1}{2}||u_{\alpha g}-z_{d}||^{2}_{H}+\frac{M}{2}||g||^{2}_{H},\quad \forall g\in H
\]
and taking $g=0$, we obtain that there exists a positive constant $C_{1}$ such that
\[
\frac{1}{2}||u_{\alpha g_{\alpha}}-z_{d}||^{2}_{H}+\frac{M}{2}||g_{\alpha}||^{2}_{H}\leq \frac{1}{2}||u_{\alpha 0}-z_{d}||^{2}_{H}\leq C_{1}
\]
because $\{u_{\alpha 0}\}$ is convergent when $\alpha \to \infty$, see \cite[Theorem 7]{GMOT}. Therefore, there exist positive constants $C_{2}$ and $C_{3}$, independent of $\alpha$, such that
\begin{equation}\label{cota}
||g_{\alpha}||_{H}\leq C_{2}\quad \text{and}\quad
||u_{\alpha g_{\alpha}}||_{H}\leq C_{3}.
\end{equation}
Now, we choose $v = u_{\infty g^{*}} - u_{\alpha g_{\alpha}} \in V_0$
as a test function in the elliptic boundary hemivariational inequality (\ref{Pj0alfavariacional}) to obtain
\begin{equation*}
a(u_{\alpha g_{\alpha}}, u_{\infty g^{*}} - u_{\alpha g_{\alpha}})
+ \alpha \int_{\Gamma_{3}}
j^0(u_{\alpha g_{\alpha}}; u_{\infty g^{*}} - u_{\alpha g_{\alpha}}) \, d\Gamma
\ge L(u_{\infty g^{*}} - u_{\alpha g_{\alpha}}).
\end{equation*}
From the equality
$$a(u_{\alpha g_{\alpha}}, u_{\infty g^{*}} - u_{\alpha g_{\alpha}}) =
- a(u_{\infty g^{*}} - u_{\alpha g_{\alpha}}, u_{\infty g^{*}} - u_{\alpha g_{\alpha}})
+ a(u_{\infty g^{*}}, u_{\infty g^{*}} - u_{\alpha g_{\alpha}}),$$
we get
\begin{equation}\begin{split}\label{N2}
& a(u_{\infty g^{*}} - u_{\alpha g_{\alpha}}, u_{\infty g^{*}} - u_{\alpha g_{\alpha}}) - \alpha \int_{\Gamma_{3}}
j^0(u_{\alpha g_{\alpha}}; u_{\infty g^{*}} - u_{\alpha g_{\alpha}}) \, d\Gamma \\ & \le a(u_{\infty g^{*}}, u_{\infty g^{*}} - u_{\alpha g_{\alpha}}) - L(u_{\infty g^{*}} - u_{\alpha g_{\alpha}}).
\end{split}\end{equation}
Taking into account that
$j^0(x, u_{\alpha g_{\alpha}}; u_{\infty g^{*}} - u_{\alpha g_{\alpha}}) = j^0(x, u_{\alpha g_{\alpha}}; b-u_{\alpha g_{\alpha}})$
on $\Gamma_{3}$, and by $H(j)$(d), we have
$j^0(x, u_{\alpha g_{\alpha}}; u_{\infty g^{*}} - u_{\alpha g_{\alpha}}) \le 0$ on $\Gamma_{3}$.
Hence
\begin{equation*}
a(u_{\infty g^{*}} - u_{\alpha g_{\alpha}},u_{\infty g^{*}} - u_{\alpha g_{\alpha}}) \le a(u_{\infty g^{*}}, u_{\infty g^{*}} - u_{\alpha g_{\alpha}}) - L(u_{\infty g^{*}} - u_{\alpha g_{\alpha}}).
\end{equation*}
By the boundedness and coerciveness of $a$, we infer
\begin{equation*}
m_a \| u_{\infty g^{*}} - u_{\alpha g_{\alpha}} \|_V^2 \le (M_{a} \|u_{\infty g^{*}}\|_V
+ \| L \|_{V^*}) \, \| u_{\infty g^{*}} - u_{\alpha g_{\alpha}} \|_V
\end{equation*}
with $M_{a} > 0$,
and subsequently
\begin{equation}\begin{split}\label{ZZZ}
\| u_{\alpha g_{\alpha}} \|_V &\le \| u_{\infty g^{*}} - u_{\alpha g_{\alpha}} \|_V + \| u_{\infty g^{*}} \|_V\\ &
\le \frac{1}{m_a}
(M_{a} \|u_{\infty g^{*}}\|_V  + \| L \|_{V^*}) +
\| u_{\infty g^{*}} \|_V \\& =: C_{4},
\end{split} \end{equation}
where $C_{4} > 0$ is a constant independent of $\alpha$.
Hence, since $a(u_{\infty g^{*}} - u_{\alpha g_{\alpha}},u_{\infty g^{*}} - u_{\alpha g_{\alpha}}) \ge 0$, from (\ref{N2}),
we have
\begin{equation*}\begin{split}
- \alpha \int_{\Gamma_{3}}
j^0(u_{\alpha g_{\alpha}}; u_{\infty g^{*}} - u_{\alpha g_{\alpha}}) \, d\Gamma &\le
(M_{a} \|u_{\infty g^{*}}\|_V  + \| L \|_{V^*}) \, \| u_{\infty g^{*}} - u_{\alpha g_{\alpha}} \|_V\\ &
\le
\frac{1}{m_a} (M_{a} \|u_{\infty g^{*}}\|_V  + \| L \|_{V^*})^2
\\ &=: C_5,
\end{split}\end{equation*}
where $C_5 > 0$ is independent of $\alpha$.
Thus
\begin{equation}\label{N3}
- \int_{\Gamma_{3}}
j^0(u_{\alpha g_{\alpha}}; u_{\infty g^{*}} - u_{\alpha g_{\alpha}}) \, d\Gamma \le
\frac{C_5}{\alpha}.
\end{equation}
It follows from (\ref{ZZZ}) that
$\{ u_{\alpha g_{\alpha}} \}$ remains in a bounded subset
of $V$. Thus,
there exists $\eta \in V$ such that,
by passing to a subsequence if necessary, we have
\begin{equation}\label{CONV8}
u_{\alpha g_{\alpha}} \rightharpoonup \eta \ \ \mbox{weakly in} \ \ V, \ \mbox{as} \
\alpha \to \infty.
\end{equation}
Moreover. from (\ref{cota}) we have that there exists $h\in H$ such that
\begin{equation}\label{convcontrol}
g_{\alpha} \rightharpoonup h\ \ \mbox{weakly in} \ \ H, \ \mbox{as} \
\alpha \to \infty.
\end{equation}
\newline
\textbf{Step 2.} Next, we will show that $h=g^{*}$ and $\eta =u_{\infty g^{*}}$. We observe that $\eta \in V_0$ because
$\{u_{\alpha g_{\alpha}} \} \subset V_0$ and $V_0$ is sequentially weakly closed in $V$. Let $w \in K$ and $v = w - u_{\alpha g_{\alpha}} \in V_0$. From (\ref{Pj0alfavariacional}), we have
\begin{equation*}
L(w-u_{\alpha g_{\alpha}}) \le a(u_{\alpha g_{\alpha}}, w -u_{\alpha g_{\alpha}})
+ \alpha \int_{\Gamma_{3}}
j^0(u_{\alpha g_{\alpha}}; w-u_{\alpha g_{\alpha}}) \, d\Gamma.
\end{equation*}
Since $w = b$ on $\Gamma_3$, by $H(j)$(d), we have
\begin{equation*}
\alpha \int_{\Gamma_{3}} j^0(u_{\alpha g_{\alpha}}; w-u_{\alpha g_{\alpha}}) \, d\Gamma
= \alpha \int_{\Gamma_{3}} j^0(u_{\alpha g_{\alpha}}; b-u_{\alpha g_{\alpha}}) \, d\Gamma \le 0
\end{equation*}
which implies
\begin{equation}\label{INEQ2}
L(w-u_{\alpha g_{\alpha}}) \le a(u_{\alpha g_{\alpha}}, w -u_{\alpha g_{\alpha}}).
\end{equation}
Next, we use the weak lower semicontinuity of the functional
$V \ni v \mapsto a(v,v) \in \real$ and from (\ref{INEQ2}),
we deduce
\begin{equation}\label{INEQ4}
\eta \in V_0 \ \ \mbox{satisfies} \ \
L(w-\eta) \le a(\eta, w-\eta) \ \ \mbox{for all} \ \ w \in K.
\end{equation}
Subsequently, we will show that $\eta \in K$.
In fact, from (\ref{CONV8}), by the compactness of the trace operator, we have
$u_{\alpha g_{\alpha}} \big|_{\Gamma_{3}} \to
\eta \big|_{\Gamma_{3}}$ in $L^2(\Gamma_3)$, as
$\alpha \to \infty$. Passing to a subsequence if necessary,
we may suppose that
$u_{\alpha g_{\alpha}} (x) \to \eta(x)$ for a.e. $x \in \Gamma_3$
and there exists $f \in L^2(\Gamma_3)$ such that
$|u_{\alpha g_{\alpha}}(x)| \le f(x)$ a.e. $x \in \Gamma_3$.
Using the upper semicontinuity of the function
$\real\times \real \ni (r, s) \mapsto j^0(x, r; s)
\in \real$ for a.e. $x \in \Gamma_3$,
see \cite[Proposition 3 (iii)]{GMOT}, we get
$$
\limsup_{\alpha\rightarrow \infty} j^0(x, u_{\alpha g_{\alpha}}(x); b - u_{\alpha g_{\alpha}}(x))
\le j^0(x,\eta(x); b-\eta(x))
\ \ \mbox{a.e.} \ \ x \in \Gamma_3.
$$
Next, taking into account the estimate
\begin{equation*}
|j^0(x, u_{\alpha g_{\alpha}}(x); b - u_{\alpha g_{\alpha}}(x))|
%= |j^0(u_{\alpha g_{\alpha}}(x); b - u_{\alpha g_{\alpha}}(x))|
\le (c_0 + c_1 |u_{\alpha g_{\alpha}} (x)|) \, |b-u_{\alpha g_{\alpha}}(x)| \le k(x)
\ \ \mbox{a.e.} \ \ x \in \Gamma_3
\end{equation*}
with $k \in L^1(\Gamma_3)$ given by
$k(x) = (c_0 + c_1 f(x)) (|b| + f(x))$,
by the dominated convergence theorem, see~\cite{DMP}, we
obtain
$$
\limsup_{\alpha\rightarrow \infty} \int_{\Gamma_3} j^0(u_{\alpha g_{\alpha}}; b - u_{\alpha g_{\alpha}})
\, d\Gamma
\le \int_{\Gamma_3} j^0(\eta; b-\eta)\, d\Gamma.
$$
Consequently, from $H(j)$(d) and (\ref{N3}), we have
\begin{equation*}
0 \le -\int_{\Gamma_{3}} j^0(\eta; b-\eta) \, d\Gamma
\le \liminf_{\alpha\rightarrow \infty} \left(
-\int_{\Gamma_{3}} j^0(u_{\alpha g_{\alpha}}; b-u_{\alpha g_{\alpha}})
\, d\Gamma \right) \le 0
\end{equation*}
which gives
$\int_{\Gamma_{3}} j^0(\eta; b-\eta) \, d\Gamma =0$.
Again by $H(j)$(d), we get $j^0(x, \eta; b-\eta) = 0$
a.e. $x\in\Gamma_{3}$.
Using $(H_1)$, we have $\eta(x) = b$ for a.e.
$x \in \Gamma_3$, which together with (\ref{INEQ4}) implies
\begin{equation*}\label{INEQ5}
\eta \in K \ \ \mbox{satisfies} \ \
L(w-\eta) \le a(\eta, w-\eta) \ \ \mbox{for all} \ \ w \in K.
\end{equation*}
Next, we will prove that $\eta = u_{\infty h}$. To this end, let $v := w -\eta \in K_0$ with
arbitrary $w \in K$.
Hence, $L(v) \le a(\eta, v)$ for all $v \in K_0$.
Recalling that $v \in K_0$ implies $-v \in K_0$,
we obtain
$a(\eta, v) \le L(v)$ for all $v \in K_0$.
Hence, we conclude that
$$
\eta \in K \ \ \mbox{satisfies} \ \
a(\eta, v) = L(v)\ \ \mbox{for all} \ \ v \in K_0,
$$
i.e.,
$\eta \in K$ is a solution to problem (\ref{Pvariacional}).
By the uniqueness of solution to problem (\ref{Pvariacional}),
we have $\eta = u_{\infty h}$ and hence
$u_{\alpha g_{\alpha}} \rightharpoonup u_{\infty h}$ weakly in $V$,
as $\alpha \to \infty$.
\newline
Now
\[
J_{\alpha}(g_{\alpha})\leq J_{\alpha}(f), \quad \forall f \in H
\]
next
\begin{equation*}\begin{split}
J(h)& =\frac{1}{2}||u_{\infty h}-z_{d}||_{H}^{2}+\frac{M}{2}||h||_{H}^{2}= \frac{1}{2}||\eta -z_{d}||_{H}^{2}+\frac{M}{2}||h||_{H}^{2}\\
&\leq \liminf_{\alpha \to \infty}J_{\alpha}(g_{\alpha})\leq \liminf_{\alpha \to \infty}J_{\alpha}(f)\\ &
=\lim_{\alpha \to \infty}J_{\alpha}(f)=J(f), \quad \forall f\in H
\end{split}\end{equation*}
and from the uniqueness of the optimal control problem (\ref{OPVariational}), see \cite{GaTa}, we obtain that
\[
h=g^{*},
\]
therefore $u_{\infty h}=u_{\infty g^{*}}$. Next, we have that, when $\alpha \to \infty$
\begin{equation*}
g_{\alpha}\rightharpoonup g^{*} \ \ \text{weakly in} \ \ H \quad \text{and}\quad
u_{\alpha g_{\alpha}} \rightharpoonup u_{\infty g^{*}} \ \ \text{weakly in} \ \ V .
\end{equation*}
\newline
\textbf{Step 3.} Now, we prove the strong convergence
$u_{\alpha g_{\alpha}} \to u_{\infty g^{*}}$ in $V$, as $\alpha \to \infty$.
Choosing
$v = u_{\infty g^{*}}-u_{\alpha g_{\alpha}} \in V_0$ in problem (\ref{Pj0alfavariacional}),
we obtain
\begin{equation*}
a(u_{\alpha g_{\alpha}}, u_{\infty g^{*}} -u_{\alpha g_{\alpha}})
+ \alpha \int_{\Gamma_{3}}
j^0(u_{\alpha g_{\alpha}}; u_{\infty g^{*}}-u_{\alpha g_{\alpha}}) \, d\Gamma
\ge L(u_{\infty g^{*}}-u_{\alpha g_{\alpha}}).
\end{equation*}
Hence
\begin{equation*}\begin{split}
a(u_{\infty g^{*}} -u_{\alpha g_{\alpha}}, u_{\infty g^{*}} -u_{\alpha g_{\alpha}})
& \le a(u_{\infty g^{*}}, u_{\infty g^{*}} -u_{\alpha g_{\alpha}})
+ L(u_{\alpha g_{\alpha}} - u_{\infty g^{*}})
\\& + \alpha \int_{\Gamma_{3}}
j^0(u_{\alpha g_{\alpha}}; u_{\infty g^{*}}-u_{\alpha g_{\alpha}}) \, d\Gamma.
\end{split}\end{equation*}
Since $u_{\infty g^{*}}= b$ on $\Gamma_3$,
by $H(j)$(d) and the coerciveness of the form $a$,
we have
$$
m_a \, \| u_{\infty g^{*}} -u_{\alpha g_{\alpha}} \|^2_V \le
a(u_{\infty g^{*}}, u_{\infty g^{*}} -u_{\alpha g_{\alpha}})
+ L(u_{\alpha g_{\alpha}} - u_{\infty g^{*}}).
$$
Employing the weak continuity of $a(u_{\infty g^{*}}, \cdot)$, the compactness of the trace operator and taking into account that $u_{\alpha g_{\alpha}}\rightarrow u_{\infty g^{*}}$ strongly in $H$, we conclude that
$u_{\alpha g_{\alpha}} \to u_{\infty g^{*}}$ strongly in $V$, as $\alpha \to \infty$.
\newline
Finally, we prove the strong convergence of $g_{\alpha}$ to $g^{*}$ in $H$, when $\alpha\rightarrow \infty$. In fact, from  $u_{\alpha g_{\alpha}}\rightarrow u_{\infty g^{*}}$ strongly in $H$, we deduce
\begin{equation}\label{a}
\lim_{\alpha\rightarrow \infty}\frac{1}{2}||u_{\alpha g_{\alpha}}-z_{d}||_{H}^{2}=\frac{1}{2}||u_{\infty g^{*}}-z_{d}||_{H}^{2}
\end{equation}
and as $g_{\alpha}\rightharpoonup g^{*}$ weakly in $H$, then
\begin{equation}\label{b}
||g^{*}||_{H}^{2}\leq \liminf_{\alpha\rightarrow \infty}||g_{\alpha}||_{H}^{2}.
\end{equation}
Next, from (\ref{a}) and (\ref{b}), we obtain
\begin{equation*}
\frac{1}{2}||u_{\infty g^{*}}-z_{d}||_{H}^{2}+\frac{M}{2}||g^{*}||_{H}^{2}\leq \liminf_{\alpha\rightarrow \infty} \left(\frac{1}{2}||u_{\alpha g_{\alpha}}-z_{d}||_{H}^{2}+\frac{M}{2}||g_{\alpha}||_{H}^{2}\right),
\end{equation*}
that is
\begin{equation*}
J(g^{*})\leq \liminf_{\alpha\rightarrow \infty}J_{\alpha}(g_{\alpha}).
\end{equation*}
On the other hand, from the definition of $g_{\alpha}$, we have
\begin{equation*}
J_{\alpha}(g_{\alpha})\leq J_{\alpha}(g^{*})
\end{equation*}
then, taking into account that $u_{\alpha g^{*}}\rightarrow u_{\infty g^{*}}$ strongly in $H$, see~\cite[Theorem~7]{GMOT}, we obtain
\begin{equation*}
\limsup_{\alpha\rightarrow \infty}J_{\alpha}(g_{\alpha})\leq \limsup_{\alpha\rightarrow \infty}J_{\alpha}(g^{*})=J(g^{*})
\end{equation*}
and therefore
\begin{equation*}
\lim_{\alpha\rightarrow \infty}J_{\alpha}(g_{\alpha})=J(g^{*})
\end{equation*}
or equivalently
\begin{equation}\label{c}
\lim_{\alpha\rightarrow \infty}\left(\frac{1}{2}||u_{\alpha g_{\alpha}}-z_{d}||_{H}^{2}+\frac{M}{2}||g_{\alpha}||_{H}^{2}\right)=\frac{1}{2}||u_{\infty g^{*}}-z_{d}||_{H}^{2}+\frac{M}{2}||g^{*}||_{H}^{2}.
\end{equation}
Now, from (\ref{a}) and (\ref{c}), when $\alpha\rightarrow \infty$, we have
\begin{equation*}
||g_{\alpha}||_{H}^{2}\rightarrow ||g^{*}||_{H}^{2}
\end{equation*}
and as $g_{\alpha}\rightharpoonup g^{*}$ weakly in $H$, we deduce that $g_{\alpha}\rightarrow g^{*}$ strongly in $H$. This completes the proof.

\end{proof}

We remark that we can find examples of several locally Lipschitz functions $j$ which satisfies the hypothesis $H(j)$ and $(H_1)$ in ~\cite{GMOT}.

\section{Conclusions}\label{Conclusions}

We have studied a parameter optimal control problems for systems governed by elliptic boundary hemivariational inequalities with a non-monotone multivalued subdifferential boundary condition on a portion of the boundary of the domain which is described by the Clarke generalized gradient of a locally Lipschitz function. We prove an existence result for the optimal controls and we show an asymptotic result for the optimal controls and the system states, when the parameter (the heat transfer coefficient on a portion of the boundary) tends to infinity.
These results generalize for a locally Lipschitz function $j$, under the hypothesis $H(j)$ and $(H_1)$, the classical results obtained in ~\cite{GaTa} for a quadratic superpotential $j$.

\section*{Acknowledgements}

The present work has been partially sponsored by the European Union’s Horizon 2020 Research and Innovation Programme under the Marie Sklodowska-Curie grant agreement 823731 CONMECH and by the Project PIP No. 0275 from CONICET and Universidad Austral, Rosario, Argentina for the second author, and by the Project PPI No. 18/C555 from SECyT-UNRC, R\'io Cuarto, Argentina for the first author.

%%=========================================

\end{document}